\documentclass{amsart}
\usepackage{amsfonts}
\usepackage{color}

\usepackage{graphicx}

\newtheorem{thm}{Theorem}[section]

\newtheorem{lema}[thm]{Lemma}
\newtheorem{prop}[thm]{Proposition}
\theoremstyle{definition}

\theoremstyle{remark}
\newtheorem{rem}[thm]{Remark}
\theoremstyle{example}

\numberwithin{equation}{section}
\newcommand{\R}{\mathbb R}

\newcommand{\ve}{\varepsilon}

\newcommand{\lam}{\lambda}

\newcommand{\cf}{\rightarrow}

\begin{document}
\title[Eigenvalue homogenization]{Eigenvalue homogenization for quasilinear elliptic equations with different boundary conditions}

\author[J Fern\'andez Bonder, J P Pinasco, A M Salort]{Juli\'an Fern\'andez Bonder, Juan P. Pinasco, Ariel M. Salort }
\address{Departamento de Matem\'atica
 \hfill\break \indent FCEN - Universidad de Buenos Aires and
 \hfill\break \indent   IMAS - CONICET.
\hfill\break \indent Ciudad Universitaria, Pabell\'on I \hfill\break \indent   (1428)
Av. Cantilo s/n. \hfill\break \indent Buenos Aires, Argentina.}
\email[J. Fern\'andez Bonder]{jfbonder@dm.uba.ar}
\urladdr[J. Fern\'andez Bonder]{http://mate.dm.uba.ar/~jfbonder}
\email[J.P. Pinasco]{jpinasco@dm.uba.ar}
\urladdr[J.P. Pinasco]{http://mate.dm.uba.ar/~jpinasco}
\email[A.M. Salort]{asalort@dm.uba.ar}


\subjclass[2010]{35B27, 35P15, 35P30}

\keywords{Eigenvalue homogenization, nonlinear eigenvalues, order of convergence}

\begin{abstract}
We study the rate of convergence for (variational) eigenvalues of several non-linear problems involving oscillating weights and subject to different kinds of boundary conditions in bounded domains.
\end{abstract}

\maketitle

\section{Introduction}

In this work we study the asymptotic behavior as $\ve \cf 0$ of the (variational) eigenvalues of
\begin{equation} \label{ec1}
 -div(a(x,\nabla u^\ve))+V(\tfrac{x}{\ve}) |u^\ve|^{p-2}u^\ve=\lam^\ve \rho(\tfrac{x}{\ve}) |u^\ve|^{p-2}u^\ve \quad \textrm{ in } \Omega
\end{equation}
with different boundary conditions (Dirichlet, Neumann, etc.), where $\Omega\subset \R^N$ is a bounded domain, $\ve$ is a positive real number, and $\lam^\ve$ is the eigenvalue parameter.
We also consider eigenvalue dependent boundary conditions,
\begin{align*}
\begin{cases}
 -div(a(x,\nabla u^\ve))+V(\tfrac{x}{\ve}) |u^\ve|^{p-2}u^\ve=\lam^\ve \rho(\tfrac{x}{\ve}) |u^\ve|^{p-2}u^\ve &\quad \textrm{ in } \Omega,\\
a(x,\nabla u^\ve)\nu = \lam^\ve |u^\ve|^{p-2}u^\ve &\quad \textrm{ in } \partial \Omega.
\end{cases}
\end{align*}
and the Steklov problem
\begin{align*}
\begin{cases}
 -div(a(x,\nabla u^\ve))+V(\tfrac{x}{\ve}) |u^\ve|^{p-2}u^\ve=0 &\quad \textrm{ in } \Omega,\\
 a(x,\nabla u^\ve)\nu=\lam^\ve |u^\ve|^{p-2}u^\ve  &\quad \textrm{ in } \partial \Omega.
\end{cases}
\end{align*}

The weight function $\rho(x)$ is assumed to be bounded away from zero and infinity, the potential function  $V(x)$ is bounded and the operator $a(x,\xi)$ has precise hypotheses that are stated below, but the prototypical example is
\begin{equation}\label{typical}
a(x, \xi) = A(x) |\xi|^{p-2} \xi, \qquad x\in \Omega, \ \xi\in\R^N
\end{equation}
with $1<p<+\infty$, and $A(x)$ is a uniformly elliptic matrix.

The problem of finding the asymptotic behavior of the eigenvalues of \eqref{ec1} has
relevance in different fields of applications, and it is an important part of
\emph{Homogenization Theory}.

We will consider the important case of periodic homogenization, i.e., the functions
$\rho(x)$ and $V(x)$ are assume to be $Q-$periodic functions, $Q$ being the unit cube
in $\R^N$.

The natural limit problem of \eqref{ec1} as $\ve\to 0$ is given by
\begin{equation} \label{ec1.lim}
 -div(a(x,\nabla u))+\bar V|u|^{p-2}u=\lam \bar \rho |u|^{p-2}u \quad \textrm{ in } \Omega
\end{equation}
with the corresponding boundary condition, where $\bar \rho$ and $\bar V$ are the
averages of $\rho$ and $V$ in $Q$, respectively.

In this work, we focus our attention on the order of convergence of the eigenvalues, that is, an estimate of $|\lam_k^\ve - \lam_k|$ in terms of $k$ and $\ve$, where $\lam_k^\ve$ and $\lam_k$ are the $k$--th variational eigenvalues of problems \eqref{ec1} and \eqref{ec1.lim} respectively.

The homogenization problem for eigenvalues has deserved a great deal of attention in
the past, specially in the linear case, that is, problem \eqref{ec1}  with $a(x, \xi)$
given by \eqref{typical} and $p=2$.

In particular, the first result on the order of convergence for the linear problem complemented with homogeneous Dirichlet boundary conditions can be found in \cite{OSY92} where it is proved that
$$
|\lam_k^{\ve}-\lam_k| \le Ck^{\frac{6}{N}} \ve^{\frac12}.
$$
where $C$ is a positive constant independent of $k$ and $\ve$

Later on, again in the linear case and with homogeneous Dirichlet boundary conditions,
Santonsa and Vogelius \cite{SaVo93} proved, by using asymptotic expansions, that
$$
|\lam^\ve_k - \lam_k|\leq C \ve
$$
where $C$ depends on $k$.

We want to remark that in the above mentioned works, the authors allowed for an $\ve$ dependance on the diffusion matrix in \eqref{typical}.

More recently, the linear problem in dimension $N=1$, and for the operator
$a(x,\xi)=\xi$, was studied by Castro and Zuazua in \cite{CZ00, CZ00b}. In those
articles the authors, using the so-called WKB method which relays on asymptotic
expansions of the solutions of the problem, and the explicit knowledge of the
eigenfunctions and eigenvalues of the constant coefficient limit problem, proved the
bound
$$
|\lam_k^\ve - \lam_k| \le C k^4\ve.
$$
Let us mention that their method needs higher regularity on the weight $\rho$, which must belong at least to $C^2$.

Recently, Kenig, Lin and Shen \cite{KLS11} studied the linear problem in any dimension (allowing an $\ve$ dependance in the diffusion matrix of the elliptic operator) and proved that for Lipschitz domains $\Omega$ one has
$$
|\lam_k^\ve - \lam_k| \le C \ve |\log \ve|^{\frac12 + \sigma}
$$
for any $\sigma>0$, $C$ depending on $k$ and $\sigma$.

Also, the authors show that if the domain $\Omega$ is more regular ($C^{1,1}$ is
enough) they can get rid of the logarithmic term in the above estimate. However, no
explicit dependance of $C$ on $k$ is obtained in that work.

In the non-linear case without dependence on $\ve$ in $a(x,\xi)$, and $N\geq 1$, we
proved  in \cite{FBPS12} (by means of a precise order of convergence for oscillating
integrals), that
$$
|\lam_k^\ve - \lam_k|\le C k^{\frac{2p}{N}} \ve
$$
with $C$ independent of $k$ and $\ve$.

Moreover, in \cite{FBPS12b}, for the one dimensional problem, we show that the constant entering in the above estimate can be found  explicitly and, moreover, a dependence on $\ve$ on the operator $a(x,\xi)$ was treated.

Let us stress the fact that all the above mentioned works deal with the homogenous
Dirichlet boundary condition case, with the exception of the aforementioned paper
\cite{KLS11} where also it is considered Neumann and Steklov boundary conditions and
similar results as in the Dirichlet boundary condition case were found.

Some results are known in the linear case when different boundary conditions are
considered. In the one-dimensional case with $\ve$ dependence in the operator
$a(x,\xi)$  Moskov  and Vogelius \cite{MoVo97b} by using the Osborn's eigenvalues
estimate proved that
$$|\lam^\ve_k - \lam_k|\leq C \ve$$
where $C$ depends on $k$.

Recently, in \cite{Sa12} by similar methods to those in \cite{FBPS12} the rate of
convergence of the first non-trivial curve in a weighted Fucik problem with Neumann
boundary conditions was found. As corollary  it follows that
$$
|\lam_1^\ve - \lam_1| \le C \ve, \qquad  |\lam_2^\ve - \lam_2| \le C \ve
$$
where $\lam_1^\ve$ and $\lam_2^\ve$ are the first and second eigenvalues of the nonlinear equation \eqref{ec1} with Neumann boundary conditions and $C$ does not depends on $\ve$.

\bigskip

There is a large class of usually studied eigenvalue problems related to problem
\eqref{ec1} with different boundary conditions. In this paper we analyze the most
common ones.

Let us consider $\rho,V$ two $Q-$periodic functions, being $Q$ the unit cube in $\R^N$, which satisfy
\begin{align} \label{cota.rho}
  0<\rho^-\leq \rho(x) \leq \rho^+<+\infty & \quad \textrm{a.e. } \Omega\quad \text{and}\quad
V\in L^\infty(\Omega).
\end{align}
for certain constants $\rho^-<\rho^+$. We denote $\rho_\ve(x):=\rho(\tfrac{x}{\ve})$ and  $V_\ve(x):=V(\tfrac{x}{\ve})$.

\medskip

For each $\ve>0$ fixed, we define the following eigenvalue problems:

\medskip

\begin{itemize}
  \item Dirichlet problem:
\begin{align} \label{ec.eps.d}
D_\ve(\Omega):
\begin{cases}
 -div(a(x,\nabla u^\ve))+V_\ve |u^\ve|^{p-2}u^\ve=\lam^\ve \rho_\ve |u^\ve|^{p-2}u^\ve &\quad \textrm{ in } \Omega,\\
 u^\ve=0 &\quad \textrm{ in } \partial \Omega.
\end{cases}
\end{align}
  \item Neumann problem
\begin{align} \label{ec.eps.n}
N_\ve(\Omega):
\begin{cases}
 -div(a(x,\nabla u^\ve))+V_\ve |u^\ve|^{p-2}u^\ve=\lam^\ve \rho_\ve |u^\ve|^{p-2}u^\ve &\quad \textrm{ in } \Omega,\\
 a(x,\nabla u^\ve)\nu=0  &\quad \textrm{ in } \partial \Omega.
\end{cases}
\end{align}
 \item Robin problem
\begin{align} \label{ec.eps.r}
R_\ve(\Omega):
\begin{cases}
 -div(a(x,\nabla u^\ve))+V_\ve |u^\ve|^{p-2}u^\ve=\lam^\ve \rho_\ve |u^\ve|^{p-2}u^\ve &\quad \textrm{ in } \Omega,\\
 a(x,\nabla u^\ve)\nu+\beta |u^\ve|^{p-2}u^\ve=0  &\quad \textrm{ in } \partial \Omega.
\end{cases}
\end{align}
\item Non-flux problem:
\begin{align} \label{ec.eps.p}
P_\ve(\Omega):
\begin{cases}
 -div(a(x,\nabla u^\ve))+V_\ve |u^\ve|^{p-2}u^\ve=\lam^\ve \rho_\ve |u^\ve|^{p-2}u^\ve &\quad \textrm{ in } \Omega,\\
 u^\ve=\textrm{constant} &\quad \textrm{ in } \partial \Omega, \\
 \int_{\partial\Omega} a(x,\nabla u^\ve) \nu\, dS=0.
\end{cases}
\end{align}
 \item Eigenvalue dependent boundary condition
\begin{align} \label{ec.eps.t}
B_\ve(\Omega):
\begin{cases}
 -div(a(x,\nabla u^\ve))+V_\ve |u^\ve|^{p-2}u^\ve=\lam^\ve \rho_\ve |u^\ve|^{p-2}u^\ve &\quad \textrm{ in } \Omega,\\
a(x,\nabla u^\ve)\nu = \lam^\ve |u^\ve|^{p-2}u^\ve &\quad \textrm{ in } \partial \Omega.
\end{cases}
\end{align}
  \item Steklov problem
\begin{align} \label{ec.eps.s}
S_\ve(\Omega):
\begin{cases}
 -div(a(x,\nabla u^\ve))+V_\ve |u^\ve|^{p-2}u^\ve=0 &\quad \textrm{ in } \Omega,\\
 a(x,\nabla u^\ve)\nu=\lam^\ve |u^\ve|^{p-2}u^\ve  &\quad \textrm{ in } \partial \Omega.
\end{cases}
\end{align}
\end{itemize}
We consider $\Omega$ to be a bounded domain in $\R^N$, the smoothness will be precised
in each case. The operator $a(x,\xi)$ satisfy properties (H0)--(H8) given in \S 2.
Observe that in problems \eqref{ec.eps.d}--\eqref{ec.eps.s}  the operator $a(x,\xi)$
does not depends on $\ve$, the dependence on the parameter appears only in the weights
and potential functions. By $\nu$ we denote the outer unit normal vector with respect
to $\partial \Omega$. The parameter $\beta$ in the Robin problem is in $[0,\infty)$.
We observe that when $\beta=0$ it corresponds to the Neumann problem and when
$\beta=\infty$ it corresponds with the Dirichlet problem. In the boundary condition
problem \eqref{ec.eps.t} and in the Steklov problem \eqref{ec.eps.s} we require the
potential function to be strictly positive, i.e., there exists $V^->0$ such that $V^-
\le V(x)$ a.e. in $\Omega$. Observe that this requirement is not necessary in
\eqref{ec.eps.d}--\eqref{ec.eps.p} since the hypothesis of $V$ being bounded below
away from zero can be assumed without loss of generality.

In \eqref{ec.eps.d}--\eqref{ec.eps.s} the natural limit problems (as $\ve\to 0$) are the analogous ones with the weights $\rho_\ve$ and the potentials $V_\ve$ replaced by their averages in the unit cube $Q$, i.e.
$$
\bar\rho = \int_Q \rho(y)\, dy,\qquad \bar V = \int_Q V(y)\, dy.
$$

\medskip

It is not difficult to see, for any of the problems \eqref{ec.eps.d}--\eqref{ec.eps.s}, that if $\lam^\ve$ is a convergent sequence of eigenvalues as $\ve\cf 0$ then $\lam=\lim_{\ve\to 0}\lam^\ve$ is an eigenvalue of the corresponding limit problem and, up to some subsequence, the associated eigenfunctions $u^\ve$ converge weakly to an associated eigenfunction $u$ of the corresponding limit problem.

For the Dirichlet problem \eqref{ec.eps.d} this fact was proved in \cite{BCR06} (see also \cite{FBPS12} for a simplified proof of this result). The proofs for the others problems \eqref{ec.eps.n}--\eqref{ec.eps.s} are analogous.

Our aim is to study the order of convergence of the eigenvalues of problems \eqref{ec.eps.d}--\eqref{ec.eps.s} to those of the limit equations.

\medskip
Using results concerning to oscillating integrals, we prove our main results:

\begin{thm}\label{rateDNRT}
Let $\lam_k^\ve$ be the $k$--th variational eigenvalue associated to any of the problems \eqref{ec.eps.d}--\eqref{ec.eps.t}, respectively. Let $\lam_k$ be the $k$--th variational eigenvalue associated to the correspondent limit problem. Then there exists a constant $C>0$ independent of the parameters $\ve$ and $k$ such that
$$
|\lam_k^\ve - \lam_k| \leq C k^{\frac{2p}{N}}\ve.
$$
\end{thm}

\begin{thm}\label{rateS}
Let $\lam_k^\ve$ be the $k$--th variational eigenvalue associated to equation \eqref{ec.eps.s}. Let $\lam_k$ be the $k$--th variational eigenvalue associated to the correspondent limit problem. Then there exists a constant $C>0$ independent of the parameters $\ve$ and $k$ such that
$$
|\lam_k^\ve - \lam_k| \leq C k^{\frac{p-1}{N-1}}\ve.
$$
\end{thm}

\begin{rem}
Let us note that, for problem \eqref{ec.eps.t}, the bound can be improved when $p<N$,
and we get
$$
|\lam_k^\ve - \lam_k| \leq C k^{\frac{2(p-1)}{N-1}}\ve.
$$
\end{rem}

The rest of the paper is organized as follows. In Section \S 2 we introduce the class
of operators considered and the hypotheses on the functions $a(x, \xi)$, $\rho$, $V$,
and the notation which will be used. In Section \S 3 we analyze the eigenvalue
problems and we show the relationships between them. Section \S 4 is devoted to
oscillatory integrals, and in Section \S 5 we prove the main results.

\section{Preliminary results}
\subsection{Monotone operators} \label{sec.op.monot}
We start this section by making the precise assumptions on the function $a(x,\xi)$

Let $\Omega\subset \R^N$, $N\geq 1$ be a bounded domain. We consider $a\colon \Omega\times \R^N\to \R^N$ that satisfies the following conditions:

\begin{enumerate}
\item[(H0)] {\em measurability:} $a(\cdot,\cdot)$ is a Carath\'eodory function, i.e. $a(x,\cdot)$ is continuous a.e. $x\in \Omega$,  and $a(\cdot,\xi)$ is measurable for every $\xi\in\R^N$.

\item[(H1)] {\em monotonicity:} $0\le (a(x,\xi_1)-a(x,\xi_2))(\xi_1-\xi_2)$.
\index{coercive}
\item[(H2)] {\em coercivity:} $\alpha |\xi|^p \le a(x,\xi) \xi$.

\item[(H3)] {\em continuity:} $a(x,\xi)\le \beta|\xi|^{p-1}$.
\index{$p-$homogeneity}
\item[(H4)] {\em $p-$homogeneity:} $a(x,t\xi)=t^{p-1} a(x,\xi)$ for every $t>0$.

\item[(H5)] {\em oddness:} $a(x,-\xi) = -a(x,\xi)$.
\end{enumerate}

Let us introduce $\Psi(x,\xi_1, \xi_2)=a(x,\xi_1) \xi_1 +a(x,\xi_2) \xi_2$ for all $\xi_1, \xi_2 \in \R^N$, and all $x\in \Omega$; and let $\delta=min\{p/2, (p-1)\}$.

\begin{enumerate}
\item[(H6)] {\em equi-continuity:}
$$
|a(x,\xi_1) -a(x,\xi_2)| \le c \Psi(x,\xi_1, \xi_2)^{(p-1-\delta)/p}(a(x,\xi_1) -a(x,\xi_2)) (\xi_1-\xi_2)^{\delta/p}
$$
\index{cyclical monotonicity}
\item[(H7)] {\em cyclical monotonicity:} $\sum_{i=1}^k  a(x,\xi_i) (\xi_{i+1}-\xi_i) \le 0$, for all $k\ge 1$, and $\xi_1,\ldots, \xi_{k+1}$, with $\xi_1=\xi_{k+1}$.

\item[(H8)] {\em strict monotonicity:} let $\gamma = \max(2,p)$, then
$$
\alpha |\xi_1-\xi_2|^{\gamma}\Psi(x,\xi_1,\xi_2)^{1-(\gamma/p)}\le (a(x,\xi_1)-a(x,\xi_2)) (\xi_1-\xi_2).
$$
\end{enumerate}

Hypotheses (H1)--(H3) are necessary to ensure the $G-$convergence of the operators associated to $a(x,\xi)$. On the other hand, hypotheses (H4)--(H7) are all important in the context of a well-posed eigenvalue problem. We assume (H8) for technical reasons.

We add that the conditions (H0)--(H8) are not completely independent of each other. It can be seen easily that (H8) implies (H1)--(H2) and that (H4) implies (H3) in addition to the continuity of the coefficient, for details see \cite{BCR06}.

\medskip

\begin{rem}
The prototype for such functions is $a(x,\xi)=A(x)|\xi|^{p-2} \xi$, where $A(\cdot)$ is a measurable function with values in the set of $N\times N$ symmetric matrices which satisfies
$$\alpha' |\xi|^2 \leq A(x)\xi\cdot \xi, \quad |A(x) \xi| \leq \beta' |\xi| \quad \forall \xi\in \R^N, \textrm{ a.e. } x\in\Omega.$$
for some positive constants $\alpha'$ and $\beta'$.
\end{rem}

\medskip

In particular, under these conditions, we have the following Proposition due to Baffico, Conca and Rajesh \cite{BCR06}

\begin{prop}\label{potential.f}
Given $a(x,\xi)$ satisfying conditions {\em (H0)--(H8)}, there exists a unique
Carath\'eodory function $\Phi$ which is even, $p-$homogeneous strictly convex and
differentiable in the variable $\xi$ satisfying
\begin{equation}\label{cont.coer.phi}
\alpha |\xi|^p \le \Phi(x,\xi)\le \beta |\xi|^p
\end{equation}
for all $\xi\in \R^N$ a.e. $x\in\Omega$ such that
$$
\nabla_{\xi} \Phi(x,\xi)= p a(x,\xi)
$$
and normalized such that $\Phi(x,0)=0$.
\end{prop}
\begin{proof}
See Lemma 3.3 in \cite{BCR06}.
\end{proof}

\subsection{General hypotheses and notation}

Throughout the paper the following hypotheses and notation will be used:
\begin{itemize}
\item By $Q$ we always mean the unit cube in $\R^N$, i.e. $Q=[0,1]^N$.

\item The functions $g, \rho, V$ will always refer to functions in $L^\infty(\R^N)$ that are $Q-$periodic.

\item For $\ve>0$ we denote
$$
g_\ve(x) = g(\tfrac{x}{\ve}),\qquad \rho_\ve(x) = \rho(\tfrac{x}{\ve}),\qquad V_\ve(x) = V(\tfrac{x}{\ve}).
$$

\item The average on $Q$ will be denoted by
$$
\bar g = \int_Q g(y)\, dy, \qquad \bar \rho = \int_Q \rho(y)\, dy,\qquad \bar V  = \int_Q V(y)\, dy.
$$

\item Observe that $g_\ve\rightharpoonup \bar g$, $\rho_\ve\rightharpoonup \bar \rho$ and $V_\ve\rightharpoonup \bar V$ (as $\ve\to 0$) weakly *  in $L^\infty$.

\item On the weight function $\rho$ we assume that it is bounded away from zero. That is, there exist constants $0<\rho^- < \rho^+<\infty$ such that
$$
\rho^- \le \rho(x)\le \rho^+.
$$

\item On the potential function $V$, for the problems \eqref{ec.eps.t} and
    \eqref{ec.eps.s}, we assume that the potential function is strictly positive,
    i.e., there exists $V^->0$ such that
    $$V^- \le V(x) \; \mbox{ a.e. in } \; \Omega.$$
    \item Lipschitz regularity of the  domain $\Omega$ is assumed in problems
        \eqref{ec.eps.d} and  \eqref{ec.eps.p}, and $C^1$ regularity in all other
        cases.
\end{itemize}

\section{Eigenvalues}
In this section we define the variational eigenvalues of problems
\eqref{ec.eps.d}--\eqref{ec.eps.s} and collect some properties and relations between
them. We focus only on the properties of the eigenvalues that will be used later, a
more detailed study can be found in the following references: the Dirichlet problem
  \eqref{ec.eps.d} was studied by Garc\'ia--Azorero and Peral Alonso in \cite{GaPa}; the Neumann problem
can be found in the work of Huang \cite{Lin}; the Robin problem \eqref{ec.eps.r} can be
found in \cite{Le}, together with the  Non-flux problem \eqref{ec.eps.p}; both
generalize periodic and separated boundary conditions in classical Sturm Liouville
problems, studied by several authors, see the paper of Binding and Rynne \cite{BR}
among others; the problem with Eigenvalue dependent boundary conditions
\eqref{ec.eps.t} was studied by Binding, Browne and Watson \cite{BBW} in the one
dimensional case,  and the
 Steklov problem \eqref{ec.eps.s} was considered by Fern\'andez Bonder and Rossi in \cite{FBR}.
Finally, more general monotone operators than the $p-$Laplacian, like the ones we will
consider here, were studied by Kawohl, Lucia and Prashanth, see \cite{KLP07}.

First of all, we observe that by replacing $\lam^\ve$ by $\lam^\ve + \|V\|_{\infty} + V_-$ in \eqref{ec.eps.d}--\eqref{ec.eps.p} we can assume that the potential function verifies that
$$
V(x)\ge V_->0
$$
In problems \eqref{ec.eps.t}--\eqref{ec.eps.s} this has to be imposed on $V$.

By means of the  Ljusternik-Schnirelmann theory (see \cite{Sz88} for instance) we know that the variational spectrum of these problems consists in countable sequences of positive eigenvalues tending to $+\infty$. Define the followings functionals
\begin{align*}
&F(u,\rho)=\int_\Omega \rho |u|^p, \\
&G(u)=\int_\Omega  \Phi(x,\nabla u), \\
&H(u)=\int_{\partial\Omega}   |u|^p,
\end{align*}
where $\Phi(x,\xi)$ is the potential function given in Proposition \ref{potential.f}.
Then, for each fixed $\ve>0$ we can give the characterization of the $k$--th variational eigenvalues of \eqref{ec.eps.d}--\eqref{ec.eps.s} as follow:
\begin{align} \label{carac.autov}
\begin{split}
&\lam_k^D = \inf_{C\in \tilde\Gamma_k} \sup_{u \in C} \frac{G(u)+F(u,V)}{F(u,\rho)},\\
&\lam_k^N = \inf_{C\in \Gamma_k} \sup_{u \in C} \frac{G(u)+F(u,V)}{F(u,\rho)}, \\
&\lam_k^R = \inf_{C\in \Gamma_k} \sup_{u \in C} \frac{\beta H(u)+G(u)+F(u,V)}{F(u,\rho)}, \\
&\lam_k^P = \inf_{C\in \bar\Gamma_k} \sup_{u \in C} \frac{G(u)+F(u,V)}{F(u,\rho)},\\
&\lam_k^B = \inf_{C\in \Gamma_k} \sup_{u \in C} \frac{G(u)+F(u,V)}{H(u)+F(u,\rho)},\\
&\lam_k^S = \inf_{C\in \Gamma_k} \sup_{u \in C} \frac{G(u)+F(u,V)}{H(u)}.
\end{split}
\end{align}
Here,
\begin{align*}
&\Gamma_k=\{C\subset W^{1,p}(\Omega) : C \textrm{ compact, } C=-C, \,  \, \gamma(C)\geq k\},\\
&\tilde \Gamma_k=\{C\subset W^{1,p}_0(\Omega) : C \textrm{ compact, } C=-C, \,  \, \gamma(C)\geq k\},\\
&\bar \Gamma_k=\{C\subset W^{1,p}_0(\Omega)\oplus\R : C \textrm{ compact, } C=-C, \,  \, \gamma(C)\geq k\}
\end{align*}
and $\gamma(C)$ is the Kranoselskii genus (see for instance \cite{Ra74,ZSP03} for definition and properties).

From the variational characterization, we immediately obtain the following inequalities
\begin{equation}\label{BNRPD}
\lam_k^B\le \lam_k^N\le \min\{\lam_k^P,\lam_k^R\} \le \max\{\lam_k^P,\lam_k^R\} \le \lam_k^D
\end{equation}
\begin{equation}\label{BS}
\lam_k^B\le \lam_k^S
\end{equation}

\medskip

In the followings Lemmas we give upper bounds for the eigenvalues $\lam_k$ defined in \eqref{carac.autov} in terms of $k$ and $\Omega$. Here and in all the paper we will consider that $\Omega\subset \R^N$ is a bounded domain. These estimates will be useful to prove the main results since it provide us with the growth rate of the eigenvalues.

\begin{lema} \label{lema.comp.NDRT}
Let $\lam_k^N$, $\lam_k^P$, $\lam_k^D$, $\lam_k^B$ and $\lam_k^R$ be the $k$--th variational eigenvalues defined in \eqref{carac.autov}. Then
$$\lam_k^B\leq \lam_k^N\leq \min\{\lam_k^P,\lam_k^R\} \leq \max\{\lam_k^P,\lam_k^R\} \leq \lam_k^D\leq  C k^{p/N}$$
where $C$ depends only  on  $\Omega$ and the bounds \eqref{cota.rho}, \eqref{cont.coer.phi}.

\end{lema}
\begin{proof} From \eqref{BNRPD}, it is enough to prove the last inequality.

Now, from \eqref{cont.coer.phi} we have
$$
\frac{G(u)+ F(u,V)}{F(u,\rho)}\le \frac{\max\{\beta,V^+\}}{\rho^-} \frac{\int_\Omega |\nabla u|^p+|u|^p}{\int_\Omega |u|^p},
$$
from where it follows that
\begin{align*}
\lam_k^D\le \frac{\max\{\beta,V^+\}}{\rho^-}\mu_k,
\end{align*}
where $\mu_k$ is the $k-$th eigenvalue of
\begin{equation} \label{ecuay}
\begin{cases}
  -\Delta_p u+|u|^{p-2}u=\mu  |u|^{p-2}u &\quad \textrm{ in } \Omega \\
   u =0 &\quad \textrm{ on } \partial \Omega.
 \end{cases}
\end{equation}

Observe that $u\in W^{1,p}_0(\Omega)$ is solution of \eqref{ecuay} if and only if $u$
is solution of
\begin{equation*}
\begin{cases}
  -\Delta_p u=\tilde \mu  |u|^{p-2}u &\quad \textrm{ in } \Omega \\
   u =0 &\quad \textrm{ on } \partial \Omega,
 \end{cases}
\end{equation*}
where $\tilde \mu = \mu-1$, which satisfies that (see \cite{GAP88})
\begin{align} \label{cc2}
\tilde \mu_k \le C k^{p/N},
\end{align}
and the result follows.
\end{proof}

\begin{lema} \label{lema.comp.S}
Let $\lam_k^B$ and $\lam_k^S$ be the $k$--th variational eigenvalue of $B(\Omega)$ and $S(\Omega)$ respectively. Then
$$\lam_k^B\leq \lam_k^S\leq C k^{\tfrac{p-1}{N-1}},$$
where $C$ is a constant depending on $V^+$, $\alpha$ and $\Omega$.
\end{lema}
\begin{proof}
From \eqref{cont.coer.phi} and \eqref{cota.rho} we have
$$
\frac{G(u)+F(u,V)}{H(u)}\leq \max\{\tfrac{1}{\alpha},V^+\} \frac{\int_\Omega |\nabla u|^p + \int_\Omega |u|^p}{\int_{\partial\Omega} |u|^p}
$$
from where it follows that
\begin{align} \label{lema.comp.S.ec1}
\lam_k^S\le \max\{\tfrac{1}{\alpha},V^+\} \mu_k,
\end{align}
where $\mu_k$ is the $k-$th eigenvalue of
\begin{align*}
\begin{cases}
 -\Delta_p u+ |u|^{p-2}u=0 &\quad \textrm{ in } \Omega\\
 |\nabla u|^{p-2}\frac{\partial u}{\partial \eta}=\mu |u|^{p-2}u  &\quad \textrm{ in } \partial \Omega.
\end{cases}
\end{align*}
Ii is proved in \cite{Pi07} the following estimate for $\mu_k$
\begin{equation} \label{lema.comp.S.ec2}
\mu_k \leq c k^{\frac{p-1}{N-1}}.
\end{equation}
where $c$ is a positive constant independent of $k$.

Finally, from \eqref{BS}, \eqref{lema.comp.S.ec1} and \eqref{lema.comp.S.ec2} the
result follows.
\end{proof}

\begin{rem}
From the previous lemmata, we have that
$$\lam_k^B\leq \min\{C k^{\tfrac{p}{N}}, C k^{\tfrac{p-1}{N-1}}\},$$
equivalently,
$$\lam_k^B\leq \left\{ \begin{array}{ll} C k^{\tfrac{p}{N}} & p\ge N,\\
 C k^{\tfrac{p-1}{N-1}} & p\le N.\end{array} \right.$$
\end{rem}

\section{Preliminaries on oscillatory integrals.}
In order to deal with the rate of convergence of the eigenvalues, the main tool that we use  is the study of oscillating integrals. These will allow us to replace an integral involving a rapidly oscillating function with one that involves its average in the unit cube.

Let $\rho$ be a $Q$-periodic weight. It is well-known that $\rho(\tfrac{x}{\ve})$ converges weakly* in $L^\infty$ to its average over $Q$. We are interested in the rate of the convergence in terms of $\ve$. In \cite{FBPS12} it is proved that
\begin{thm}[\cite{FBPS12}, Theorem 5.5]\label{teo_n_dim}
Let $\Omega\subset \R^N$ be a bounded domain with Lipschitz boundary and let $g\in L^\infty(\R^N)$ be a $Q-$periodic function, $Q$ being the unit cube in $\R^N$. Then, there exists a constant $C$ depending only on $p$, $\Omega$ and $\|g\|_{L^\infty(\R^N)}$ such that
\begin{align*}
\Big|\int_{\Omega} (g(\tfrac{x}{\ve}) - \bar g) |u|^p\Big|  \le C \ve \|\nabla u\|_{L^{p}(\Omega)}^p,
\end{align*}
for every $u\in W^{1,p}_0(\Omega)$, where $\bar g$ is the average of $g$ over $Q$.
\end{thm}

With the aid of Theorem \ref{teo_n_dim}, in \cite{FBPS12} we were able to analyze the Dirichlet boundary condition case. To deal with different boundary conditions, we need a similar Theorem that allows us to include the function space $W^{1,p} (\Omega)$.

The fact of enlarge the set of test functions is reflected in the need for more regularity on the domain $\Omega$.  In \cite{Sa12} the following result is proved.

\begin{thm}[\cite{Sa12}, Theorem 4.3] \label{teo.neu}
Let $\Omega\subset \R^n$ be a bounded domain with $C^1$ boundary and let $g\in L^\infty(\R^N)$ be a $Q-$periodic function, $Q$ being the unit cube in $\R^N$. Then for every $u\in W^{1,p}(\Omega)$ there exists a constant $C$ depending only on $p$, $\Omega$ and $\|g\|_{L^\infty(\R^N)}$ such that
$$
\left| \int_\Omega (g(\tfrac{x}{\ve})  - \bar{g}) u \right| \leq C\ve \|u\|_{W^{1,p}(\Omega)}.
$$
\end{thm}

\bigskip

Now, we need a couple of simple technical lemmas which are used in the proof of our main results.

\begin{lema} \label{lema0}
Let $u\in W^{1,p}(\Omega)$. Then
$$\||u|^p\|_{W^{1,1} (\Omega)} \leq p\|u\|_{W^{1,p}(\Omega)}^p \leq C(G(u)+F(u,\bar \rho))$$
\end{lema}
where $\rho$ is an arbitrary weight satisfying \eqref{cota.rho} and $C$ is a constant depending on $p$, $\rho$ and the bounds of \eqref{cont.coer.phi}.
\begin{proof}
  By using Young's   inequality
\begin{align} \label{lema0.ec1}
\begin{split}
 \||u|^p\|_{W^{1,1} (\Omega)} &= \||u|^p\|_{L^1(\Omega)} + p \| |u|^{p-1} \nabla u\|_{L^1(\Omega)}   \\
 &\leq p\|u\|_{L^p(\Omega)}^p + \|\nabla u\|_{L^p(\Omega)}^p\\
 &\leq p\|u\|_{W^{1,p}(\Omega)}^p.
\end{split}
\end{align}
Moreover, by \eqref{cont.coer.phi}
\begin{align} \label{lema0.ec2}
\begin{split}
\|u\|^p_{W^{1,p}(\Omega)} &\leq
 \frac{1}{\bar \rho} \left(\bar \rho \int_\Omega |u|^p +\frac{\bar \rho}{\alpha}\int_\Omega \Phi(x,\nabla u)\right)\\
 &\leq \frac{1}{\bar \rho} \max\{\tfrac{\bar \rho}{\alpha},1\}(G(u)+F(u,\bar\rho)).
\end{split}
\end{align}
From \eqref{lema0.ec1} and \eqref{lema0.ec2}, the Lemma follows.
\end{proof}

\begin{lema} \label{lema1}
Let $u\in W^{1,p}(\Omega)$ and $\rho, V\in L^\infty(\R^N)$ be two $Q-$periodic functions satisfying  \eqref{cota.rho}. Then there exists a constant $c$ depending only on $p$, $\|\rho\|_{L^\infty(\R^N)}$, $\|V\|_{L^\infty(\R^N)}$, the constants in \eqref{cota.rho} and \eqref{cont.coer.phi} such that
$$
\frac{F(u,\bar \rho)}{F(u,\rho_\ve)}\leq 1+c\ve\frac{F(u,\bar V)+G(u)}{F(u,\bar \rho)}
$$
and
$$
\frac{F(u,\rho_\ve)}{F(u,\bar\rho)}\leq 1+c\ve\frac{F(u,V_\ve)+G(u)}{F(u,\rho_\ve)}.
$$
\end{lema}

\begin{proof}
Applying  Theorem \ref{teo.neu} we obtain that
\begin{align} \label{lema1.ec1}
  \frac{\bar{\rho}\int_{\Omega} |u|^p}{\int_{\Omega} \rho_\ve|u|^p} \leq 1+ C\ve \frac{\| |u|^p\|_{W^{1,1}(\Omega)}}{\int_\Omega \rho_\ve |u|^p},
\end{align}
By Lemma \ref{lema0} and \eqref{cota.rho} we bound \eqref{lema1.ec1} as
\begin{align}
 1+ C\ve\frac{\bar\rho}{\rho^-} \frac{F(u,\bar V)+G(u)}{F(u,\bar\rho)}.
\end{align}

Similarly, by Lemma \ref{lema0} and \eqref{cota.rho} we get
\begin{align*}
  \frac{\int_{\Omega} \rho_\ve|u|^p}{\bar{\rho} \int_{\Omega} |u|^p} &
  \leq 1+ C\ve \frac{\| |u|^p\|_{W^{1,1}(\Omega)}}{\bar\rho \int_\Omega  |u|^p} \\
  & \leq
1+ C\ve \max\{1,\bar V\}\frac{\rho^-}{\bar\rho} \frac{1}{V^-}\frac{G(u)+F(u,V_\ve)}{F(u,\rho_\ve)}.
\end{align*}
This completes the proof of the Lemma.
\end{proof}

\section{Main results.}

The proofs of Theorems \ref{rateDNRT} and \ref{rateS} follow the same general lines of
the ones of Theorem 5.6 in \cite{FBPS12}. The fundamental tool to estimate the rates
of convergence of the eigenvalues is the error bound of oscillating integrals given in
Theorems \ref{teo_n_dim} and \ref{teo.neu}. Observe that the regularity of the domain
$\Omega$ considered in equations \eqref{ec.eps.d}--\eqref{ec.eps.s} are the necessary
to apply Theorems  \ref{teo.neu} and \ref{teo_n_dim}, i.e., Lipschitz regularity in
equation \eqref{ec.eps.d} and $C^1$ regularity in all other cases.

\subsection{Proof of Theorem \ref{rateDNRT}} The Dirichlet boundary condition case, i.e. problem $D_\ve(\Omega)$, was treated in \cite[Theorem 5.6]{FBPS12}.
We prove the result in detail for problem $N_\ve(\Omega)$. For $R_\ve(\Omega)$, $P_\ve(\Omega)$and $B_\ve(\Omega)$ the proofs are very similar and we will make a sketch highlighting only the differences.

For simplicity we will denote $\lam_k^\ve$ and $\lam_k$ (without the superindex $N$) to the $k$--th variational eigenvalue of $N_\ve(\Omega)$ and its limit problem obtained as $\ve\cf 0$.

Let $\delta>0$ and let $G_\delta^k\subset W^{1,p}(\Omega)$ be a compact, symmetric set of genus $k$ such that
\begin{equation} \label{ecuacion.n0}
 \lam_k =  \sup_{u \in G_\delta^k} \frac{ G(u)+ F(u,\bar{V})}{F(u,\bar\rho)} + O(\delta).
\end{equation}

We use now the set $G_\delta^k$, which is admissible in the variational
characterization of the $k$--th eigenvalue $\lam_k^\ve$, in order to find a bound for
it as follows,
\begin{align} \label{ecuacion.n1}
 \lam_k^\ve  \leq   \sup_{u \in G_\delta^k} \frac{G(u)+F(u,V_\ve)}{F(u,\bar \rho) } \frac{F(u,\bar \rho)}{F(u,\rho_\ve)}.
\end{align}
Now, we look for bounds of the two quotients in \eqref{ecuacion.n1}.

For every function $u\in G_\delta^k\subset W^{1,p}(\Omega)$ we can apply Theorem
\ref{teo.neu} and we obtain that
\begin{align} \label{ecuacion.n2}
\frac{G(u)+F(u,V_\ve)}{F(u,\bar \rho)} &\leq
\frac{G(u)+F(u,\bar V)}{F(u,\bar \rho)} + C\ve\frac{\||u|^p\|_{W^{1,1}(\Omega)}}{F(u,\bar \rho)}.
\end{align}
By using Lemma \ref{lema0}, we have for each $u\in G_\delta^k$ there exists some
constant $c>0$ such that
\begin{align} \label{ecuacion.n3}
\begin{split} \frac{\||u|^p\|_{W^{1,1}(\Omega)}}{F(u,\bar\rho)} &
 \leq C \frac{G(u)+F(u,\bar V)}{F(u,\bar\rho)} \\
 & \leq C \sup_{v \in G_\delta^k} \frac{G(v)+F(v,\bar V)}{F(v,\bar\rho)}
 \\  & =C(\lam_k + O(\delta)).
\end{split}
\end{align}
Since $u\in G_\delta^k\subset W^{1,p}(\Omega)$,  by applying  Lemma \ref{lema1}  and
\eqref{ecuacion.n3} we obtain that
\begin{align} \label{ecuacion.n4}
\begin{split}
  \frac{F(u,\bar{\rho})}{F(u,\rho_\ve)} &\leq 1+C\ve\frac{\||u|^p\|_{W^{1,1}(\Omega)}  }{F(u,\bar\rho)}  \\
&\leq 1+ C\ve (\lam_k +O(\delta)).
\end{split}
\end{align}
Then, combining \eqref{ecuacion.n2}, \eqref{ecuacion.n3} and  \eqref{ecuacion.n4} we
find that
$$
   \lam_k^\ve \leq \left(   \lam_k + O(\delta) +C\ve(  \lam_k + O(\delta) )\right) \left( 1+ C \ve (  \lam_k + O(\delta))  \right).
$$
Letting $\delta\to 0$ we get
\begin{align} \label{ecuacion.n5}
  \lam_k^\ve -   \lam_k \leq C\ve (  \lam_k^2+  \lam_k).
\end{align}
In a similar way, interchanging the roles of $  \lam_k$ and $  \lam_k^\ve$, we obtain
\begin{align} \label{ecuacion.n6}
  \lam_k -   \lam_k^\ve \leq C\ve ((  \lam_k^\ve)^2+  \lam_k^\ve).
\end{align}

So, from \eqref{ecuacion.n5} and \eqref{ecuacion.n6}, we arrive at
$$
|  \lam_k^\ve -   \lam_k|\le C\ve \max\{  \lam_k^2+ \lam_k, (  \lam_k^\ve)^2+  \lam_k^\ve\}.
$$

In order to complete the proof of the Theorem, we need an estimate on $\lam_k$ and $\lam_k^\ve$. By Lemma \ref{lema.comp.NDRT} we can compare them with the $k$--th variational eigenvalue of the $p-$Laplacian obtaining
$$
|  \lam_k^\ve -   \lam_k|\le C\ve k^{2p/N},
$$
and  the proof is complete. \qed

\medskip

In the remaining of this subsection, we highlight the difference between the Neumann case and the rest of the boundary conditions with the exception of the Steklov problem that has a separate treatment.

\medskip
\textbf{Dirichlet}:
As we mentioned before, this problem was addressed in \cite[Theorem 5.6]{FBPS12}. Here, functions are taken in $W^{1,p}_0(\Omega)$ instead $W^{1,p}(\Omega)$ in the variational characterization of the eigenvalues. This leads to use Theorem \ref{teo_n_dim} instead of Theorem \ref{teo.neu} to estimate the oscillating integrals.
Now, for each function $u\in W^{1,p}_0(\Omega)$ we can apply Theorem \ref{teo_n_dim} and obtain an analogous equation to \eqref{ecuacion.n2}
\begin{align*}
\frac{G(u)+F(u,\rho_\ve)}{F(u,\bar \rho)} &\leq
\frac{G(u)+F(u,\bar V)}{F(u,\bar \rho)} + C\ve\frac{\|\nabla u\|^p_{L^p(\Omega)}}{F(u,\bar \rho)}.
\end{align*}

Now, the difference with the Neumann case is the way we bound the quotients $\|\nabla
u\|^p_{L^p(\Omega)}/F(u,\bar \rho) $ and $\|\nabla u\|^p_{L^p(\Omega)}/F(u,\rho_\ve)$.

By \eqref{cota.rho} and \eqref{cont.coer.phi} we get
\begin{equation*}
\begin{split}
\frac{\|\nabla u\|_{L^{p}(\Omega)}^p}{F(u,\rho_\ve)}&\le \frac{\bar\rho}{\rho^-}\frac{\|\nabla u\|_{L^p(\Omega)}^p}{F(u,\bar\rho)}\\
&\le \frac{\bar\rho}{\rho^-} \frac{1}{\alpha}\frac{G(u)+F(u,\bar V)}{F(u,\bar \rho)}
\end{split}
\end{equation*}
and
\begin{equation*}
\begin{split}
\frac{\|\nabla u\|_{L^{p}(\Omega)}^p}{F(u,\bar \rho)}\le
\frac{\rho^+}{\bar \rho}\frac{\|\nabla u\|_{L^{p}(\Omega)}^p}{F(u,\rho_\ve)}.
\end{split}
\end{equation*}
Taking into account this changes, the proof is analogous to the Neumann one.

\medskip
\textbf{Non-Flux}:
Let $u\in W^{1,p}_0(\Omega)\oplus \R$, then $u=v+c$ where $v\in W^{1,p}_0(\Omega)$ and $c$ is a constant depending on $u$. It follows that $u-c \in W^{1,p}_0(\Omega)$. Observe that if $u\in W^{1,p}_0(\Omega)\oplus \R$ then $|u|^p\in W^{1,1}_0(\Omega)\oplus \R$ and then $|u|^p-c \in W^{1,1}_0(\Omega)$. By using Theorem \ref{teo_n_dim} and Theorem \ref{teo.neu} together with \eqref{cota.rho} we observe that
\begin{align} \label{d.t0}
\begin{split}
\int_\Omega (\rho_\ve -\bar \rho)|u|^p &= \int_\Omega (\rho_\ve -\bar \rho) (|u|^p-c)+\int_\Omega (\rho_\ve -\bar \rho)c \\
&\leq C\ve \|\nabla u\|_{L^p(\Omega)}^p + C\ve\|c\|_{W^{1,p}(\Omega)} \\
&\leq C\ve(\|\nabla u\|_{L^p(\Omega)}^p  +|c|)  .
\end{split}
\end{align}

Observe that $u\in W^{1,p}_0(\Omega)\oplus \R=\{u\in W^{1,p}(\Omega) : u=c  \textrm{ on } \partial \Omega \textrm{ with } c \in\R\}$. If $u=v+c\in W^{1,p}_0(\Omega)\oplus \R $, by the Sobolev trace inequality it follows that
\begin{equation} \label{d.t}
  \|u\|_{W^{1,p}(\Omega)}^p \geq c_T \|u\|_{L^p(\partial \Omega)}^p = c_T |c |^p |\partial \Omega|^p.
\end{equation}
Moreover,
\begin{equation} \label{d.t2}
   \|u\|_{W^{1,p}(\Omega)}^p \geq \|\nabla u\|_{L^p(\Omega)}^p.
\end{equation}
Then, by \eqref{d.t} and \eqref{d.t2} it follows that
\begin{equation} \label{d.t3}
   \|\nabla u\|_{L^p(\Omega)}^p + |c | \leq C\|u\|_{W^{1,p}(\Omega)}^p.
\end{equation}

From \eqref{d.t0}, \eqref{d.t3} and  Lemma \ref{lema0} it follows that if $u\in W^{1,p}_0(\Omega)\oplus \R$ then
$$
\int_\Omega (\rho_\ve -\bar \rho)|u|^p \leq C\|u\|_{W^{1,p}(\Omega)}^p \leq C(G(u)+F(u,V))
$$
where $V$ is an arbitrary weight.

Taking into account this remark, the proof in the non-flux case is analogous to the Neumann one.
Observe that for this problem, we only need the boundary $\partial \Omega$ to be Lipschitz since we have used Theorem \ref{teo_n_dim}.

\medskip

\textbf{Robin}:
The extra term $\beta H(u)$ is irrelevant in the proof. This case is completely analogous to the Neumann and it follows by means of \eqref{cota.rho}, \eqref{cont.coer.phi}, Theorem \ref{teo.neu}, Lemma \ref{lema0} and Lemma \ref{lema1}.

\medskip
\textbf{Eigenvalue depending boundary condition}: From Theorem \ref{teo.neu}, Lemma
\ref{lema0} and \eqref{cota.rho} we have that for every $u\in W^{1,p}(\Omega)$
\begin{align*}
\frac{F(u,\bar \rho)+H(u)}{F(u,\rho_\ve)+H(u)} & \leq 1+C\ve\frac{ \|u\|_{W^{1,p}(\Omega)}}{F(u,\rho_\ve)+H(u)} \\
& \leq  1+C\ve\frac{ G(u)+F(u,\bar V)}{F(u,\bar \rho)+H(u)}
\end{align*}
and
\begin{align*}
\begin{split}
\frac{G(u)+F(u,V_\ve)}{H(u)+F(u,\bar\rho)}&\leq \frac{G(u)+F(u,\bar V)}{H(u)+F(u,\bar\rho)}+C\ve\frac{ \|u\|_{W^{1,p}(\Omega)}}{H(u)+F(u,\bar\rho)} \\ &\leq (1+C\ve)\frac{G(u)+F(u,\bar V)}{H(u)+F(u,\bar\rho)}.
\end{split}
\end{align*}
Observe that \eqref{cota.rho} and \eqref{cont.coer.phi} allow us to bound in the opposite sense, that is, for each $u\in W^{1,p}(\Omega)$,
\begin{align*}
\frac{F(u,\rho_\ve)+H(u)}{F(u,\bar \rho)+H(u)}\leq  1+C\ve\frac{ G(u)+F(u,V_\ve)}{F(u,\rho_\ve)+H(u)},
\end{align*}
\begin{align*}
\begin{split}
\frac{G(u)+F(u,\bar V)}{H(u)+F(u,\rho_\ve)}\leq (1+C\ve)\frac{G(u)+F(u, V_\ve)}{H(u)+F(u,\rho_\ve)}.
\end{split}
\end{align*}

Having changed this slight detail, the proof is analogous to the Neumann case.

\subsection{Proof of Theorem \ref{rateS}}
The proof of the Steklov case is simpler due to that in this case there is only an oscillating weight in the equation $S_\ve(\Omega)$.

Let $\delta>0$ and let $G_\delta^k\subset W^{1,p}(\Omega)$ be a compact, symmetric set of genus $k$ such that
\begin{align} \label{ecuacion.sn0}
 \lam_k =  \sup_{u \in G_\delta^k} \frac{G(u)+F(u,\bar V)}{H(u)} + O(\delta).
\end{align}
Being $G_\delta^k$  admissible in the variational characterization of $\lam_k^\ve$, we get
\begin{align} \label{ecuacion.sn1}
 \lam_k^\ve  \leq   \sup_{u \in G_\delta^k} \frac{G(u)+F(u, V_\ve)}{H(u)}.
\end{align}
For every function $u\in G_\delta^k\subset W^{1,p}(\Omega)$ we can apply Theorem \ref{teo.neu} obtaining
\begin{align} \label{ecuacion.sn2}
\frac{G(u)+F(u, V_\ve)}{H(u)}\leq
\frac{G(u)+F(u, \bar V)}{H(u)} + C\ve\frac{\||u|^p\|_{W^{1,1}(\Omega)}}{H(u)}.
\end{align}

Now, for each $u\in G_\delta^k$ we can apply Lemma \ref{lema0}  to bound
\begin{align} \label{ecuacion.sn3}
\begin{split}
 \frac{\||u|^p\|_{W^{1,1}(\Omega)}}{H(u)}& \leq c \frac{G(u)+F(u,\bar V)}{H(u)} \\
 & \leq c \sup_{v \in G_\delta^k} \frac{G(v)+F(v,\bar V)}{H(v)} \\
 & =c(\lam_k + O(\delta)).
\end{split}
\end{align}
Then, combining \eqref{ecuacion.sn1}, \eqref{ecuacion.sn2} and  \eqref{ecuacion.sn3} we find that
$$
   \lam_k^\ve \leq \lam_k + O(\delta) +  C \ve (  \lam_k + O(\delta)).
$$
In a similar way, interchanging the roles of $  \lam_k$ and $  \lam_k^\ve$, we can obtain the opposite inequality. Letting $\delta\to 0$ we arrive at
$$
|  \lam_k^\ve -   \lam_k|\le C\ve \max\{ \lam_k, \lam_k^\ve\}.
$$
By Lemma \ref{lema.comp.S} we obtain that
$$
|\lam_k^\ve -   \lam_k|\le C\ve k^\frac{p-1}{N-1}
$$
and the proof is complete.

\section*{Acknowledgements}

This work was partially supported by Universidad de Buenos Aires under grant
20020100100400 and by CONICET (Argentina) PIP 5478/1438.

\bibliographystyle{amsplain}
\bibliography{biblio-tesis}

\end{document}